\documentclass[preprint,12pt]{elsarticle}

\usepackage{dsfont}
\usepackage[latin1]{inputenx}
\usepackage{amsfonts}
\usepackage{latexsym}
\usepackage{enumerate}
\usepackage{amsthm,amsmath,amssymb,mathrsfs,yhmath}
\usepackage{tikz}
\usepackage{fullpage}
\usepackage{wasysym}
\usetikzlibrary{positioning}
\usetikzlibrary{matrix}
\usepackage{float}
\usepackage{subfig}

\usepackage[pdftex,bookmarks=true]{hyperref}
\usepackage{url}

\newcommand{\cP}{\mathcal{P}}
\newcommand{\cC}{\mathcal{C}}
\newcommand{\TT}{\mathbb{T}}
\newcommand{\End}{\operatorname{End}}

\newcommand{\add}{\operatorname{add}}

\newtheorem*{thm}{Theorem} 
\newtheorem{thm1}{Theorem}[section]  
\newtheorem{lemma}[thm1]{Lemma}   
\newtheorem{coroll}[thm1]{Corollary}
\newtheorem*{corollary}{Corollary}
\newtheorem{proposition}[thm1]{Proposition}

\theoremstyle{remark}
\newtheorem*{acknowledgements}{Acknowledgements}

\theoremstyle{definition}
\newtheorem{dfn}[thm1]{Definition}

\newtheorem{rem}[thm1]{Remark}
\newtheorem{notation}[thm1]{Notation}

\journal{Journal of Algebra}

\begin{document}
\begin{frontmatter}
\title{Jacobian algebras with periodic module category and exponential growth}
\author{Yadira Valdivieso-D\'iaz}
 \ead{valdivieso@mdp.edu.ar}
\address{Dean Funes 3350,
Departamento de Matem\'atica,
Facultad de Ciencias Exactas y Naturales,
Universidad Nacional de Mar del Plata,
Argentina}

\begin{abstract}
Recently it was proven by Geiss, Labardini-Fragoso and Sh\"oer in \cite{GLFS13} that every Jacobian algebra associated to a triangulation of
a closed surface $S$ with a collection of marked points $M$ is tame and Ladkani proved in \cite{Lad12} these algebras are (weakly) symmetric. In this work we show that for these algebras the
Auslander-Reiten translation acts 2-periodically on objects. Moreover, we show that excluding only the case of a sphere with $4$ (or less) punctures, these algebras are of exponential growth. These results imply that the existing characterization of symmetric tame algebras whose non-projective indecomposable modules are $\Omega$-periodic, has at least a missing class (see \cite[Theorem 6.2]{ES08} or \cite{ES06}).

As a consequence of the 2-periodical actions of the Auslander-Reiten translation on objects, we have that
the Auslander-Reiten quiver of the generalized cluster category $\cC_{(S,M)}$
consists only of stable tubes of rank $1$ or $2$.
\end{abstract}

\end{frontmatter}

\section{Introduction}
Let $k$ be an algebraically closed field.
A potential $W$ for a quiver $Q$ is, roughly speaking, a linear combination of
cyclic paths in the  complete path algebra $k\langle\langle Q\rangle\rangle$.
The Jacobian algebra $\cP(Q,W)$ associated to a quiver with a potential
$(Q,W)$ is  the quotient of the complete path algebra $k\langle\langle Q\rangle\rangle$ modulo the Jacobian
ideal $J(W)$. Here, $J(W)$ is the topological closure closure of the ideal of $k\langle\langle Q\rangle\rangle$
which is generated by the cyclic derivatives of $W$ with respect to the arrows
of  $Q$.

Quivers with potential were introduced in\cite{DWZ08} in order to construct additive
categorifications of cluster algebras with skew-symmetric exchange
matrix. For the just mentioned categorification it is crucial that
the potential for $Q$ be non-degenerate, i.e. that it can be mutated along
with the quiver arbitrarily, see \cite{DWZ08} for more details on quivers with
potentials.

In \cite{FST08} the authors introduced, under some mild hypothesis, for each oriented
surface with marked points $(S,M)$ a mutation finite cluster algebra with
skew symmetric exchange matrices. More precisely, each triangulation $\TT$ of
$(S,M)$ by tagged
arcs corresponds to a cluster and the corresponding exchange matrix is
conveniently coded into a quiver $Q(\TT)$.
Labardini-Fragoso in \cite{LF09} enhanced this construction by introducing potentials
$W(\TT)$ and showed that these  potentials are compatible with mutations.
In particular, these potentials are non-degenerate. Ladkani showed that for
surfaces with empty boundary and a triangulation $\TT$ which has no
self-folded triangles the Jacobian algebra $\cP(Q(\TT),W(\TT))$ is symmetric 
(and in particular finite-dimensional).
It follows that for any triangulation $\TT'$ of a closed surface $(S,M)$ the
Jacobian algebra $\cP(Q(\TT'), W(\TT'))$ is weakly symmetric by \cite{HI11}.
In \cite{GLFS13} it is shown by a degeneration argument that these algebras are tame.

Next, following Amiot \cite[Sec. 3.4]{Ami12} and
Labardini-Fragoso \cite[Theorem 4.2]{LF13} we have a 2-Calabi-Yau triangulated category $\cC_{(S,M)}$ together with a family of cluster tilting
objects $(T_{\TT})_{\TT \text{ triangulation of } (S,M)}$, related by Iyama-Yoshino
mutations such that $\End_{\cC_{(S,M)}}(T(\TT)) \cong\cP(Q(\TT),W(\TT))^{\text{op}}$ (see for example \cite{Ke08} for details of $n$-Calabi-Yau categories).

\begin{thm}
Let $S$ be a closed oriented surface with a non-empty finite collection $M$ of
punctures, excluding only the case of a sphere with $4$ (or less) punctures.
For an arbitrary tagged triangulation $\TT$, the Jacobian algebra
$\cP(Q(\TT),W(\TT))$ is symmetric, tame, its stable Auslander-Reiten
quiver  consists only of stable tubes of rank $1$ or $2$ and it is an algebra of exponential growth.
\end{thm}

Recall that for symmetric algebras there is a relation between the Auslander-Reiten translation $\tau$ and the Heller translate or syzygy $\Omega$, namely $\tau \cong  \Omega^2$ (see \cite[Section 2.5]{Ga80}), then  if $\Lambda_{\TT}=\cP(Q(\TT), W(\TT))$ is a Jacobian algebra as in the main Theorem, we have that the non-projective indecomposable $\Lambda_{\TT}$-modules are $\Omega$-periodic. Related to symmetric tame algebras, Erdmann and Skowro\'nski have announced  the following statement (see \cite[Theorem 6.2]{ES08} or \cite{ES06}): {\em a non-simple indecomposable symmetric algebra $\Lambda$ is tame with $\Omega$-periodic modules if and only if $\Lambda$ belongs to one of the following classes of algebras: a representation-finite symmetric algebra, a non domestic symmetric algebra of polynomial growth, or an algebra of quaternion type (in the sense of \cite{Erd90})}. We show at the end of this work that the main Theorem implies that the existing characterization has at least a missing class.

As another consequence of the Theorem we have the following result.

\begin{corollary}
Let $S$ be a closed oriented surface with a non-empty finite collection $M$ of 
punctures, excluding only the case of a sphere with $4$ (or less) punctures.The Auslander-Reiten quiver of the generalized cluster category
$\cC_{(S,M)}$ consists only of stable tubes of rank $1$ or $2$.
\end{corollary}

\begin{notation}
Let $(S,M)$ be a marked surface with empty boundary and $\TT$
be an tagged triangulation of $(S, M)$. We construct  the
unreduced signed adjacency quiver $\widehat{Q}(\TT)$ of the
triangulation $\TT$ and following \cite{LF09} we construct the
unreduced potential $\widehat{W}(\TT)$.  The quiver with potential
$(Q(\TT), W(\TT))$ associated to the triangulation $\TT$ of the
marked surface $(S, M)$ is the reduced part of
$(\widehat{Q}(\TT), \widehat{W}(\TT))$.

Let $\mathcal C$ be the generalized cluster category of $(\mathcal
S, M)$. We denote by $\Sigma$ the suspension functor of $\mathcal
C$ and by $\Sigma_n$ the suspension functor in a $n$-angulated
category $(\mathcal F, \Sigma_n,\pentagon)$ (cf. \cite{GKO13} for
definition of $n$-angulated categories).
\end{notation}

\section{Proof of the results}

We show first that the stable Auslander-Reiten
quiver  of a Jacobian algebra $\cP(Q(\TT),W(\TT))$  of an arbitrary tagged triangulation $\TT$  of a closed oriented surface $S$ consists only of stable tubes of rank $1$ or $2$ (see \cite[Chapter VII]{ARS97} for details of Auslander-Reiten quiver). In order to prove it, we
establish two preliminary results for 2-Calabi-Yau tilted symmetric algebras. Recall that any finite dimensional Jacobian algebra is a 2-Calabi-Yau tilted algebra (see \cite[Corollary 3.6]{Ami09}).

\begin{proposition}\label{lema}
Let $\mathcal C$ be a Hom-finite 2-Calabi-Yau triangulated category and $T\in \cC$ be a cluster-tilting object. If $\End_\cC(T)$ is symmetric, then its stable Auslander-Reiten
quiver  consists only of stable tubes of rank $1$ or $2$.
\end{proposition}

\begin{proof}
By \cite[Theorem I 2.4]{RVB02} there is a relation between the Serre functor $\mathbb S$, the Auslander-Reiten translation $\tau_\cC$ and the suspension functor $\Sigma$ of $\cC$ on objects, that is,  $\mathbb S=\Sigma\tau_\cC$. Since $\cC$ is 2-Calabi-Yau, we have $\Sigma\tau=\Sigma^2$, therefore
$\tau_\cC=\Sigma$.
Denote by $\Lambda$ the algebra $\End_\cC(T)$ and by $\tau_\Lambda$ the Auslander-Reiten translation of $\mod \Lambda$.
By hypothesis,  $\Lambda$ is a symmetric algebra, then by \cite[Lema]{Ri08} the subcategory add($T$) is closed under the Serre functor $\tau_\cC^2=\mathbb S=\Sigma^2$.
Then by \cite[Remark 6.3]{GKO13}, the stable category
\underline{mod}($\Lambda$) is a 3-Calabi-Yau category. Hence by
\cite[Theorem I 2.4]{RVB02} we have
$\Omega^{-1}\tau_{\Lambda}=\Omega^{-3}$, because $\Omega^{-1}$ is
the suspension functor in \underline{mod}$\Lambda$ , therefore
$\Omega^{-2}=\tau_{\Lambda}$.

On the other hand, $\Lambda$ is symmetric, then
$\Omega^2=\tau_{\Lambda}$, therefore $\Omega^4=\mathds 1_{
\textrm{\underline{mod}}\Lambda}$. Then
$\Omega^{4}=\tau_{\Lambda}^2=\mathds 1_{
\textrm{\underline{mod}}\Lambda}$.
\end{proof}

\begin{rem} Note that in proof of Proposition \ref{lema}, we also show that symmetric 2-Calabi-Yau tilted algebras are algebras whose
indecomposable non-projective modules are $\Omega$-periodic and the period is a divisor of $4$.
\end{rem}

\begin{proposition}\label{lema2}
Let $\mathcal C$ be a Hom-finite 2-Calabi-Yau triangulated category. If there exists a cluster-tilting object $T\in \cC$ such that $\End_\cC(T)$ is symmetric,
then the Auslander-Reiten quiver of the category
$\cC$ consists only of stable tubes of rank $1$ or $2$
\end{proposition}

\begin{proof}

Let $T$ be a cluster-tilting object in $\mathcal C$ such that the 2-Calabi-Yau tilted algebra $\Lambda=\End_{\mathcal C}(T)$ is symmetric, therefore by \cite[Proposition 6.4]{GKO13}
proj$(\Lambda)=\add(T)$ is a $4$-angulated category, with the Nakayama functor
$\nu$ as suspension. It well known that the Nakayama functor of a
symmetric algebra is the identity (see \cite[Lemma I.3.5]{Erd90}).

By \cite[Remark 6.3]{GKO13}, then the suspension in the
$4$-angulated category  $\add(T)$ satisfies  $\Sigma_4=\Sigma^2$, but
also satisfies $\Sigma_4=\nu=\mathds 1_{\add(T)}$,
then $\Sigma^2=\mathds 1_{\add (T)}$. Hence
$\tau^2=\Sigma^2=\mathds 1_{\textrm{add}(T)}$.

The result follows for Proposition \ref{lema} and the equivalence
$\mathcal C/(\textrm{add}(\Sigma
T))\cong\textrm{mod}\Lambda$ proved in
\cite[Proposition, Section 2.1]{KR07}, which is induced by the functor  $F:\mathcal C \longrightarrow \textrm{mod} \Lambda$ which sends $X\in\mathcal C$ to the module $T \mapsto \textrm{Hom}(T, X)$.
\end{proof}

As consequence of Proposition \ref{lema2}, we have a result about the Auslader-Reiten quiver of the generalized cluster category of a closed surface with punctures.

\begin{coroll}\label{cor}
Let $S$ be a closed oriented surface with a non-empty finite collection $M$ of
punctures, excluding only the case of a sphere with $4$ (or less) punctures.The Auslander-Reiten quiver of the generalized cluster category
$\cC_{(S,M)}$ consists only of stable tubes of rank $1$ or $2$.
\end{coroll}

\begin{proof}
It was proved in \cite[Proposition 4.7]{Lad12} that there is a particular
triangulation $\mathbb T$ of $(\mathcal S, M)$ such that the
Jacobian algebra $\Lambda=\mathcal P(Q(\TT),W(\TT))$ is symmetric.

By \cite[Proposition 4.10]{FST08}, \cite[Theorem 7.1]{LF12} and
\cite[Theorem 3.2]{KY11}, the generalized cluster category $\mathcal
C$  is independent of the choice of the triangulation $\mathbb T$,
then the generalized cluster category $\mathcal C$ is equivalent to
the generalized cluster category $\mathcal C_{(Q(\TT),W(\TT))}$. The
result follows from Proposition \ref{lema2}.
\end{proof}

\begin{rem}\label{simetrica}
A partial converse to Proposition \ref{lema2} is given in \cite[Lemma 2.2. (c)]{BIKR08}. Moreover, this result and Corollary \ref{cor} imply that any Jacobian algebra of a tagged triangulation of a closed Riemann surface is not only weakly-symmetric, but is symmetric.
\end{rem}

Now, we prove the first part of Theorem, that is, we prove that for an arbitrary tagged triangulation $\mathbb T$, the Auslander Reiten quiver of the Jacobian algebra $\cP(Q(\TT),W(\TT))$ consists only of stable tubes of rank 1 or 2.

\begin{proof}[Proof of periodic module category]
Let $\TT$ be a triangulation of the marked surface $(\mathcal S,
M)$ and $\Lambda=\mathcal P(Q(\TT), W(\TT))$ be the Jacobian Algebra
of the triangulation $\TT$. It was already known that $\Lambda$ is a
tame algebra (see \cite{GLFS13}), and by Remark \ref{simetrica} $\Lambda$ is symmetric.  Therefore, it
only remains to prove that its stable Auslander-Reiten quiver
consists only of stable tubes of rank $1$ or $2$.  Consider a
cluster tilting object $T$ of $\mathcal C$ such that the functor  $F:\mathcal C \longrightarrow \textrm{mod} \Lambda$ which sends $X\in\mathcal C$ to the module $T \mapsto \textrm{Hom}(T, X)$ induces an equivalence $\mathcal
C/(\textrm{add}(\Sigma T))\cong\textrm{mod}\Lambda$ (see \cite[Proposition, Section 2.1]{KR07}). The functor $F$ also induces an equivalence
\underline{mod}$\Lambda\cong\mathcal C/(T, \Sigma T)$ (see \cite[Section 3.5]{KR07}). Therefore the
statement follows from Corollary \ref{cor}.
\end{proof}

Finally, before we prove the second part of the main result,  we recall the
definition of a tame algebra of exponential growth. Recall also that for the following definition the ground field $k$ is assumed to be algebraically closed.

Following Drozd \cite{D80}, an algebra $A$ is \textit{tame} if for every dimension $d\in
\mathbb N$ there is a finite number of $A-k[X]-$bimodules $N_1,
\dots, N_{i(d)}$ such that each $N_i$ is finitely generated free
over the polynomial ring $k[X]$ and almost all $d$-dimensional indecomposable
$A$-modules are isomorphic to $N_i\otimes_{k[X]}k[X]/ (X-\lambda)$
for some $i\in\{1,\dots, i(d)\}$ and some $\lambda\in k$. For a tame algebra $A$, we denote by $\mu_A(d)$ the smallest possible number of these $N_i$. Then $A$ is said to be of \emph{polynomial growth} \cite{Sk87} (respectively, \emph{domestic} \cite{CB91}) if there is a positive integer $m$ such that $\mu_A(d)\leq d^m$ (respectively, $\mu_A(d)\leq m$) for all $d\geq 1$. We say that $A$ is tame of \textit{exponential growth} if $\mu_A(d)> r^d$ for
infinitely many $d\in \mathbb N$ and some real number $r>1$.

Also, we recall the definition of string and band in string algebras. Given an arrow $\alpha:i \to j$ in a quiver $Q$, we denote by $\alpha^{-1}:j\to i$ the formal inverse of $\alpha$. Given such a formal inverse $l=\alpha^{-1}$, one writes $l^{-1}=\alpha$. Let $\bar{Q_1}$ be the set of all arrows and their formal inverses, the elements of $\bar{Q_1}$ are letters. A \emph{string} $w$, of an algebra $kQ/I$, is a sequence $l_1l_2\cdots l_n$ of the elements of $\bar{Q_1}$ such that

\begin{itemize}
\item[(W1)] We have $l_i^{-1}\neq l_{i+1}$, for all $1\leq i< n$.
\item[(W2)] No proper subsequence of $w$ or its inverse belongs to $I$.
\item[(W3)] end($l_i$) = start($l_{i+1}$) for all $i$.
\end{itemize}

If $w=l_1l_2\cdots l_n$ and $w'=l'_1l'_2\cdots l'_m$ are strings we say that the composition of $w$ and $w'$ is defined provide $l_1l_2\cdots l_nl'_1l'_2\cdots l'_m$ is a string, and write $ww'=l_1l_2\cdots l_nl'_1l'_2\cdots l'_m$.

A string $w$ is said to be \emph{cyclic} if all powers $w^m$, $m\in\mathbb N$ are strings. Given a cyclic string $w$, the powers $w^m$ with $m\geq 2$ are said to be \emph{proper powers}. A cyclic string $w$ is said to be \emph{primitive} provided it is not a proper power of some other string. A \emph{band} is a cyclic primitive string. 

\begin{rem}\label{growth}
From the proof of the \cite[Theorem 3.6]{GLFS13} follows that mutations of quivers with potential also preserve exponential growth, therefore  it is enough to show a particular triangulation of each closed surface which induces a  Jacobian algebra of exponential growth.
\end{rem}

\begin{lemma}\label{exponential}
Let $A$ be a finite dimensional algebra and $A'$ be a quotient of $A$. If $A'$ is of exponential growth, then  so is $A$.
\end{lemma}

The proof follows from the fact that any  $A'$-module is also an
$A$-module and if $L\cong N$ as  $A-k[x]-$bimodules, then $L\cong N$
as $A'-k[x]-$bimodules. Therefore $\mu_{A'}(d)\leq \mu_A(d)$.

Our goal is to find a triangulation $\TT$ for each closed surface with marked points $(S, M)$ such that there is a quotient of the Jacobian algebra $\cP(Q(\TT),W(\TT))$ which is an algebra of exponential growth. We have to distinguish in
two cases: the sphere with 5 punctures and the other closed
surfaces.

\begin{lemma}
For an arbitrary triangulation $\TT$ of a sphere with 5 punctures, the
Jacobian algebra $\cP(Q(\TT),W(\TT))$ is an algebra of exponential
growth.
\end{lemma}

\begin{proof}
Consider the skewed-gentle triangulation $\TT$ of Figure
\ref{esfera} (see definition of skewed-gentle triangulation in \cite[Section 6.7]{GLFS13}) and the Jacobian algebra $\Lambda=\mathcal{P}(Q(\TT),
W(\TT)$. The skewed-gentle triangulation $\TT$ was already studied in \cite{GLFS13}.

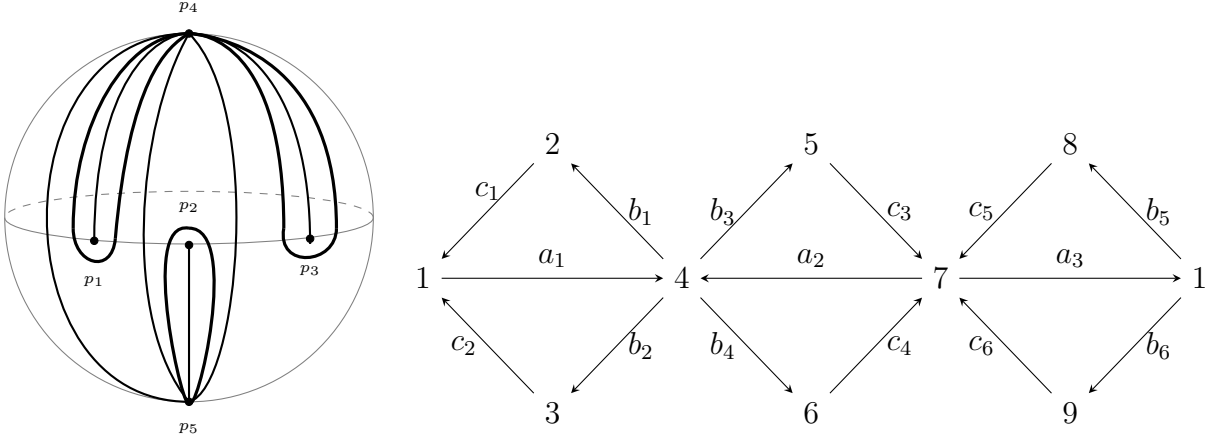
\begin{figure}[ht!]
\centering
    \subfloat{
    \begin{tikzpicture}[scale=0.7]
    \draw[gray] (3.5,3.5) circle (3.5cm);
    \draw[gray](0,3.5) arc (180:360: 3.5cm and 0.5cm);
    \draw[gray,thin,dashed] (7,3.5) arc (0:180: 3.5cm and 0.5cm);
    \draw[thick, black] (3.5,7) .. controls +(180:3.6cm) and +(180:3.6cm) .. (3.5,0);
    \draw[thick, black] (3.5,7) .. controls +(180:.65cm) and +(90:3cm) .. (1.7,3.06);
    \draw[thick, black] (3.5,7) .. controls +(0:1cm) and +(90:3cm) .. (5.8,3);
    \draw[thick, black] (3.5,3) .. controls +(-90:1cm) and +(90:2cm) .. (3.5,0);
 \draw[thick, black] (3.5,7) .. controls +(-120:2cm) and +(130:2cm) .. (3.5,0);
  \draw[thick, black] (3.5,7) .. controls +(-45:2cm) and +(15:1cm) .. (3.5,0);
   \draw[very thick, black] (3.5,7) .. controls +(0:.3cm) and +(90:3.5cm).. (5.3,3.3);
   \draw[very thick, black] (5.3,3.3) .. controls +(260:.8cm) and +(-90:.7cm).. (6.3,3.3);
  \draw[very thick, black] (6.3,3.3) .. controls +(90:3cm) and +(-5:.9cm).. (3.5,7);


    \draw[very thick, black] (3.5,7) .. controls +(180:1.2cm) and +(90:2cm).. (1.3,3.3);
    \draw[very thick, black] (1.3,3.3) .. controls +(270:.7cm) and +(-95:.7cm).. (2.1,3.1);
    \draw[very thick, black] (2.1,3.1) .. controls +(85:3cm) and +(-150:.5cm).. (3.5,7);
    %
   \draw[very thick, black] (3.5,0) .. controls +(-70:.4cm) and +(0:1cm).. (3.5,3.3);
   \draw[very thick, black] (3.5,0) .. controls +(-100:.4cm) and +(170:1cm).. (3.5,3.3);

    \filldraw [black] (3.5,7) circle (2pt)
                  (3.5,0) circle (2pt)
                  (1.7,3.06) circle(2pt)
                  (5.8,3.1) circle(2pt)
                  (3.5,2.98) circle(2pt);
\draw (3.5,7.5) node {\tiny$p_{4}$};
\draw (3.5,-.5) node {\tiny$p_{5}$};
\draw (1.7,2.3) node {\tiny$p_{1}$};
\draw (3.5,3.7) node {\tiny$p_{2}$};
\draw (5.8,2.5) node {\tiny$p_{3}$};
    \end{tikzpicture}
    }
  \subfloat{
  \begin{tikzpicture}
 \matrix (m)[matrix of math nodes, row sep=3em,column sep=3em,ampersand replacement=\&]
{\& 2\& \& 5\& \& 8\&\&\\
1\&\ \& 4\&\ \&7\&\&1\\
\&3\&\ \&6\&\&9\&\\};
  \path[-stealth]
    (m-2-1) edge node[above]{$a_1$} (m-2-3)
 (m-1-2) edge node[above]{$c_1$} (m-2-1)
(m-2-3) edge node[right]{$b_1$}(m-1-2) edge node[left]{$b_3$}(m-1-4) edge node[left]{$b_4$}(m-3-4) edge node[right]{$b_2$}(m-3-2)
(m-3-2) edge node[left]{$c_2$}(m-2-1)
(m-3-4) edge node[right]{$c_4$}(m-2-5)
(m-2-5) edge node[above]{$a_2$}(m-2-3)
(m-2-5)edge node[above]{$a_3$}(m-2-7)
(m-1-6) edge node[left]{$c_5$}(m-2-5)
(m-2-7) edge node[right]{$b_6$} (m-3-6)
(m-1-4) edge node[right]{$c_3$}(m-2-5)
(m-2-7) edge node[right]{$b_5$}(m-1-6)
(m-3-6) edge node[left]{$c_6$}(m-2-5);
\end{tikzpicture}
}
\caption{Triangulation $\TT$ of a sphere with 5 punctures}
\label{esfera}
\end{figure}

Let $I$ be the ideal of $k\langle\langle Q(\TT)\rangle\rangle$ generated
by the set $\partial(W')$ of cyclic derivatives of $W'$, where
$W'=b_5c_5a_2b_1c_1+a_1b_4c_4a_3$. Then the quotient algebra
$\Lambda'=\Lambda/ I$ is isomorphic $k\langle Q' \rangle/J$ where
$Q'$ is the quiver in Figure \ref{gentle} and $J$ is the ideal in
$k\langle Q' \rangle$ generated  by $\epsilon_i^2-\epsilon_i$,
$a_ib_i$, $b_ic_i$ and $c_ia_i$ for $i=1,2,3$ and the set
$\{b_2\epsilon_2c_2a_3$, $\epsilon_2c_2a_3a_1$, $c_2a_3a_1b_2$, $a_1b_2\epsilon_2c_2$, $\epsilon_3c_3a_2b_1\epsilon_1c_1$, $c_3a_2b_1\epsilon_1c_1b_3$, $a_2b_1\epsilon_1c_1b_3\epsilon_3$, $b_1\epsilon_1c_1b_3\epsilon_3c_3$, $\epsilon_1c_1b_3\epsilon_3c_3a_2$, $c_1b_3\epsilon_3c_3a_2b_1$, $b_3\epsilon_3c_3a_2b_1\epsilon_1\}$.
Then the quotient $\Lambda/I$ is a skewed-gentle algebra (see definition of skewed-gentle algebras in \cite{GP99}).

\begin{figure}[ht!]
\centering
\begin{tikzpicture}
 \matrix (m)[matrix of math nodes, row sep=3em,column sep=3em,ampersand replacement=\&]
{\& 4\& \& 5\& \& 6\&\&\\
1\&\ \& 2\&\ \&3\&\&1\\};
  \path[-stealth]
    (m-2-1) edge node[above]{$a_1$} (m-2-3)
 (m-1-2) edge node[above]{$c_1$} (m-2-1)
(m-2-3) edge node[right]{$b_1$}(m-1-2) edge node[left]{$b_2$}(m-1-4)
(m-2-5) edge node[above]{$a_2$}(m-2-3)edge node[above]{$a_3$}(m-2-7)
(m-1-6) edge node[left]{$c_3$}(m-2-5)
(m-1-4) edge node[right]{$c_2$}(m-2-5)
(m-2-7) edge node[right]{$b_3$}(m-1-6);
\path
 (m-1-2) edge [loop above] node {$\epsilon_1$} (m-1-2)
 (m-1-4) edge [loop above] node {$\epsilon_2$} (m-1-4)
 (m-1-6) edge [loop above] node {$\epsilon_3$} (m-1-6);
 \end{tikzpicture}
\caption{Skewed-gentle quiver $Q'$}
\label{gentle}
\end{figure}
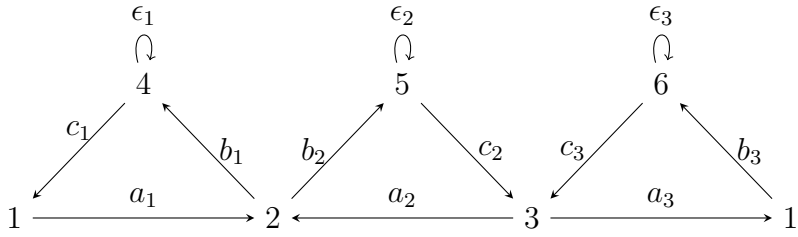

The indecomposable representations of skewed-gentle (or more generally clannish) algebras are described by a combinatorial rule in terms of certain words; similar to the more widely known special biserial algebras, see \cite{CB89} for more details. Thus in our situation it is enough to show that the corresponding clan $C=(k, Q',
S_p,(q_b), \leq)$ (see \cite[Definition 1.1]{CB89}) admits two
bands such that any arbitrary sequence of composition of them is again a band.

Let $C=(k, Q', S_p,(q_b\mid b\in S_p), \leq)$ be the clan of the
algebra $\Lambda'$, where the special loops $S_p$ are $\epsilon_1,
\epsilon_2$ and $\epsilon_3$ and with the following relations:
$a_i<b_i^{-1}$, $b_i<c_i^{-1}$ and $c_i<a_i^{-1}$ for $i=1,2,3$.

Recall, a word in a clan $C$ is a formal sequence $w_1w_2 \dots w_n$
of letters with $n>0$ such that for each $1\leq i<n$ we have end$(w_i)=\text{start}(w_{i+1})$ and the letters $w_i^{-1}$ and $w_{i+1}$ are
incomparable.

Now, consider the bands $\alpha=a_1a_2^{-1}a_3$ and
$\beta=a_1b_2\epsilon_2^*c_2c_3^{-1}\epsilon_3^*b_3^{-1}$. Observe
that $a_3^{-1}$ and $a_1$ are incomparable and also $b_3$ and $a_1$
are incomparable, then the products $\alpha\beta$ and $\beta\alpha$
are well defined and they are also bands, therefore $\Lambda'$ is an
algebra of exponential growth and by Lemma \ref{exponential} then
$\Lambda$ is so. One needs some number theory argument to prove that the string algebras or skewed-gentle algebras having two bands as before are of exponential growth, these number theory argument was given by Skowro\'nsky for a particular case in \cite[Lemma 1]{Sk87}. Finally, as we mention in Remark \ref{growth} mutations of
quiver with potential preserve exponential growth property and flips of tagged triangulations are compatible with mutations of quiver with potential (cf. \cite[Theorem 30]{LF09}), then any
Jacobian algebra associated to a triangulation of a sphere with 5
punctures is also of exponential growth.
\end{proof}

For an ideal triangulation $\TT$, the \textit{valency}
val$_{\TT}(p)$ of a puncture $p\in\mathbb P$ is the number of arcs
in $\TT$ incident to $p$, where each loop at $p$ is counted twice.

\begin{rem}\label{quivers}
Let $\TT$ be a triangulation with no self-folded triangles of a
marked surface with empty boundary $(S, \mathbb
M)$ and $(Q(\TT), W(\TT))$ be the quiver with potential associated to the triangulation $\TT$.

\begin{enumerate}
\item The quiver $Q(\TT)$ is a block-decomposable graph of blocks of type II, where a block of type II is a 3-cycle. See \cite{FST08} for details.

\item If every puncture $p\in M$ has valency at least 3, then any arrow $\alpha$ of $Q(\TT)$
has exactly two arrows $\beta,\gamma$ starting at the
terminal vertex of $\alpha$  and exactly one arrow $\delta$ ending
at the terminal vertex of $\alpha$. Following Ladkani in
\cite{Lad12} there are two functions $f, g:Q_1(\TT)\to Q_1(\TT)$  such that
$\alpha f(\alpha)f^2(\alpha)$ is a 3-cycle arising from a triangle in
$\TT$ and $(g^{n_\alpha-1}(\alpha))(g^{n_\alpha-2}(\alpha))\dots
(g(\alpha))(\alpha)$ is a cycle surrounding a puncture
$q$, where $n_\alpha=\min\{r>0 \mid g^r(\alpha)=\alpha\}$.

\item If every puncture has valency at least four, then there are no commutativity relations involving only paths of length two.
\end{enumerate}
\end{rem}

\begin{proof}[Proof of exponential growth property]

It was proven in \cite[Proposition 5.1]{Lad12}, that if the marked
surface (S,M) is not a sphere with 4 or 5 punctures, it has a
triangulation  $\TT$ with no self-folded triangles and in which each
puncture $p\in M$ has valency at least four.

Let $f, g: Q_1(\TT)\to Q_1(\TT)$ be functions as in Remark \ref{quivers} part
2). Let $I$ be the ideal in $A$ generated by the relations $\alpha
f(\alpha)$ for every arrow $\alpha$ in $Q_1(\TT)$ and consider the
quotient $A'=A/I$. It is clear that this quotient is a string
algebra.

To prove that $A'$ is an algebra of exponential growth we use the the argument as in the case of a sphere with 5 puncture, then it is enough
to prove that $A'$ admits two bands $\xi$ and $\eta$ such that any arbitrary combination of them is again a band.

Consider an arrow $\alpha:i\to j$ of the quiver $Q(\TT)$, we denote by
$\TT_{\alpha}$ the following piece of the triangulation $\TT$. The
vertices of $\alpha$ are two arcs of a triangle $\triangle_\alpha$
of $\TT$, and if $q$ and $p$ are the endpoints of the remaining arc
of this triangle, we let $\TT_{\alpha}$ the set of arcs of $\TT$
that have $q$ or $p$ as one of its endpoints. We observe that the
arc with endpoints $p$ and $q$ is a side of exactly two triangles,
one of them is the triangle $\triangle_\alpha$ and we denote by
$\triangle_\delta$ the other one. In Figure \ref{triangulation}, we
show the piece of triangulation $\TT_\alpha$.
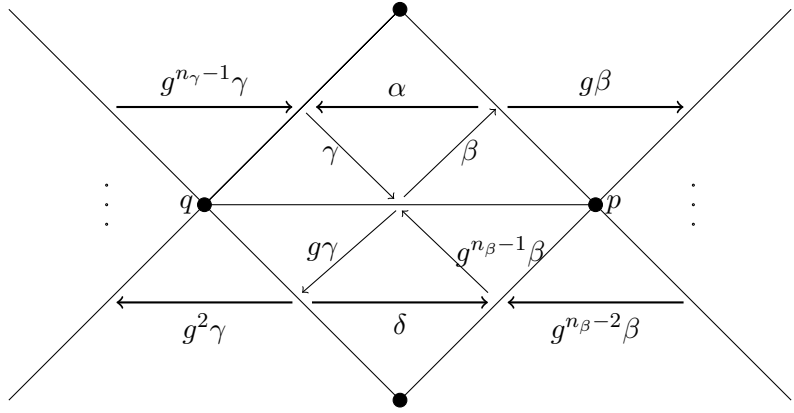
\begin{figure}[ht]
\centering
\begin{tikzpicture}[scale=2.6][
 every node/.style={align=center}]
\draw (0,0) -- (1,1) plot coordinates {(0,0) (1,1) (2,0) (0,0) (1,-1) (2,0)};
\draw(0,0)--(-1,1) plot coordinates{(0,0) (-1,-1)};
\draw(2,0)--(3,1) plot coordinates{(2,0) (3,-1)};
\draw[->](1.02,0.04)--node[right]{\small$\beta$}(1.49,0.49);
\draw[->, thick](1.4,0.5)-- node[above] {\small$\alpha$} (0.57,0.5);
\draw[->](0.52,0.47)--node[left]{\small$\gamma$}(0.97,0.03);
\draw[->,thick](0.55,-0.5)--node[below]{\small$\delta$}(1.45,-0.5);
\draw[->](1.45,-0.45)--node[right]{\small$g^{n_\beta-1}\beta$}(1.01,-0.03);
\draw[->](0.98,-0.03)--node[left]{\small$g\gamma$}(0.5,-0.45);
\foreach \x in {-0.1,0,0.1}
    \draw (-0.5,\x) circle (0.005);
\foreach \x in {-0.1,0,0.1}
    \draw (2.5,\x) circle (0.005);
\draw[->, thick](1.55,0.5)--node[above]{\small$g\beta$}(2.45,0.5);
\draw[->, thick](2.45,-0.5)--node[below]{\small$g^{n_\beta-2}\beta$}(1.55,-0.5);
\draw[->, thick](0.45,-0.5)--node[below]{\small$g^{2}\gamma$}(-0.45,-0.5);
\draw[->, thick](-0.45,0.5)--node[above]{\small$g^{n_\gamma-1}\gamma$}(0.45,0.5);
\filldraw(2,0) circle(1pt) node[right]{\small $p$};
\filldraw(0,0) circle(1pt) node[left]{\small $q$};
\filldraw(1,1) circle(1pt);
\filldraw(1,-1) circle(1pt);
\end{tikzpicture}
\caption{The piece of triangulation $\TT_{\alpha}$}
\label{triangulation}
\end{figure}

Denote by $\alpha\gamma\beta$ the 3-cycle arising from the triangle
$\triangle_\alpha$. Surrounding the puncture $q$, there is a cycle
$w(q)$ starting at $j$ which is denoted by
$$(\gamma)(g\gamma)(g^2\gamma)(g^3\gamma)\cdots(g^{n_\gamma-1}\gamma),$$
and surrounding $p$ there is a cycle $w(p)$ starting at $i$
which is denoted by
$$(g\beta)(g^2\beta)\cdots(g^{n_{\beta}-2}\beta)(g^{n_\beta-1}\beta)(\beta).$$

We denote by
$\rho_1(\alpha)$ the word
$$(g^2\gamma)(g^3\gamma)\cdots(g^{n_\gamma-1}\gamma),$$  by
$\rho_2(\alpha)$ the word
$$(g\beta)(g^2\beta)\cdots(g^{n_{\beta}-2}\beta)$$ and by
$\delta$ the arrow $fg\gamma$, which is part of the 3-cycle arising
from the triangle $\triangle_\delta$.

Then denote by $\xi(\alpha)$  the word
$$(\alpha)(\rho_1(\alpha))^{-1}(\delta)(\rho_2(\alpha))^{-1}.$$ 
Observe that $\rho_1(\alpha)$ and $\rho_2(\alpha)$ are defined for any arrow $\alpha\in Q_1$, therefore we can define $\xi(\alpha)$ for any arrow $\alpha$. In particular, we also have the string $\xi(g\beta)$. We define $\eta= \xi(g\beta)^{-1}$

Observe that both strings $\xi(\alpha)$ and $\eta$ are actually bands, moreover, we have that the compositions $\xi(\alpha)\eta$ and $\eta\xi(\alpha)$ are well defined bands and so is any word formed by (arbitrary) compositions of  the bands $\xi(\alpha)$ and $\eta$. Therefore $A'$ is an algebra of exponential growth, and by Lemma \ref{exponential} it follows that  $A$ is an algebra of exponential growth.
The result follows from the fact that exponential growth is preserved by mutations (see Remark \ref{growth}).
\end{proof}

Finally, in order to show the main Theorem shows that the existing characterization of symmetric tame algebras whose non-projective indecomposable modules are $\Omega$-periodic  was not complete, we first recall some definitions.

\begin{dfn}
Let $\Lambda$ and $\Lambda'$ be self injective algebras.
\begin{itemize}
\item The algebras $\Lambda$ and $\Lambda'$ are said to be \emph{socle equivalent} if the factor algebras $\Lambda / \textrm{Soc}(\Lambda)$ and $\Lambda' / \textrm{Soc}(\Lambda')$ are isomorphic.

\item The algebra $\Lambda$ is said to be a \emph{self injective algebra of Dynkin type} $\Delta$ if $\Lambda$ is isomorphic to an orbit algebra $\hat{B}/ G$, where $\hat{B}$ is the repetitive algebra of a tilted algebra $B$ of Dynkin type $\Delta$ and $G$ is an admissible group of automorphisms of $\hat{B}$.

\item The algebra $\Lambda$ is said to be a \emph{self injective algebra of tubular type} if  $\Lambda$ is isomorphic to an orbit algebra $\hat{B}/ G$, where $\hat{B}$ is the repetitive algebra of a tubular algebra $B$ and $G$ is an admissible group of automorphisms of $\hat{B}$.

\item The algebra $\Lambda$ is said to be of \emph{quaternion type} if the following conditions are satisfied:
\begin{itemize}
\item $\Lambda$ is symmetric, indecomposable, tame of infinite representation type.
\item The indecomposable nonprojective finite dimensional $\Lambda$-modules are $\Omega$-periodic of
period dividing 4.
\item The Cartan matrix of $\Lambda$ is non-singular.
\end{itemize}
\end{itemize}

We say that an algebra $A$ is of \emph{pure quaternion type} if it is of quaternion type and not of polynomial growth
\end{dfn}

According to the characterization of  Erdmann and Skowro\'nski (see \cite{ES06} or \cite[Theorem 6.2]{ES08}) any symmetric tame algebras with $\Omega$-periodic modules is one of the following type of algebra:

\begin{enumerate}
\item[i)] socle equivalent to a symmetric algebra of Dynkin type (or equivalently a representation finite symmetric algebra);
\item[ii)] socle equivalent to a symmetric algebra of tubular type (or equivalently a non-domestic symmetric algebra of exponential growth);
\item[iii)] an algebra of pure quaternion type.
\end{enumerate}

In the following Remark we show that any Jacobian algebra arising from a closed surface with marked points, excluding only a sphere with 4 punctures, does not belong to none of one of the previous classes of algebras.

\begin{rem}

\begin{itemize}
\item [(1)] A combination of \cite[Theorem 7.1]{GLFS13}, \cite[Theorem]{Lad12}, \cite[Proposition 2.10]{FST08} and the Theorem of this work, yields that for each closed surface $S$ of genus $g$ with $p$ punctures, excluding only the case of a sphere with 4 punctures, and each tagged triangulation $\TT$, the Jacobian algebra $\cP(Q(\TT), W(\TT))$ is symmetric and tame with $\Omega$-periodic module category and  with $6(g-1)+ 3p$ simple modules. However, $\cP(Q(\TT), W(\TT))$ is neither of the following algebras:

\begin{itemize}
\item[a)] an algebra of blocks of quaternion type, because in the case of a torus with one puncture, it was shown by Ladkani in \cite{Lad12}, that the Cartan matrix of the associated Jacobian algebra is singular, therefore it is not blocks of quaternion type by definition. In the other cases, the module category of the Jacobian algebra has at least 6 simple modules, in contrast to algebras of quaternion type which have at most 3 simple modules (see \cite[Theorem]{Erd88}).
\item[b)] an algebra socle equivalent to an algebra of tubular type, because this algebra is an algebra of polynomial growth.
\item[c)] socle equivalent to an algebra of Dynkin type, algebras which are of finite representation type.

\end{itemize}

Therefore, the existing characterization of algebras which are symmetric, tame and with the non-projective indecomposable modules $\Omega$-periodic, was not complete (see \cite[Theorem 6.2]{ES08} or \cite{ES06}). These family of Jacobian algebras forms a new family with these properties.

\item [(2)] A similar statement to the theorem is known to hold for the sphere with $4$
punctures, except
that in this case the potentials depend also on the choice of a parameter
$\lambda\in k\setminus\{0,1\}$ and in this case the Jacobian algebras
$\cP(Q(\TT),W(\TT,\lambda))$ are (weakly) symmetric of tubular type
$(2,2,2,2)$ and this is of polynomial (lineal) growth, and the category $\cC_{(Q,M,\lambda)}$ is a tubular cluster category
of type $(2,2,2,2)$, see \cite{GeGo13} and \cite{BG09}.

\item [(3)] Similar result to Theorem and Corollary was
announced by Ladkani at the abstract of the Second ARTA conference, see \cite{Lad13}.

\item[(4)] The results of this work were presented at the Second and Third ARTA conference.
\end{itemize}
\end{rem}

\begin{acknowledgements}
I want to thank Christof Geiss for suggesting me n-angulated
categories for solving this problem, for suggesting me including the
exponential growth property of this kind of algebras and valuable
comments. I also want to thank Sonia Trepode for her constant
support.
The author was partially supported by a CONICET doctoral fellowship.
\end{acknowledgements}

\bibliographystyle{elsarticle-num}
\bibliography{refe}

\begin{thebibliography}{10}
\expandafter\ifx\csname url\endcsname\relax
  \def\url#1{\texttt{#1}}\fi
\expandafter\ifx\csname urlprefix\endcsname\relax\def\urlprefix{URL }\fi
\expandafter\ifx\csname href\endcsname\relax
  \def\href#1#2{#2} \def\path#1{#1}\fi

\bibitem{GLFS13}
C.~Geiss, D.~Labardini-Fragoso, J.~Schr\"oer, The representation type of
  \textsc{J}acobian algebras arXiv:1308.0478.

\bibitem{Lad12}
S.~Ladkani, On \textsc{J}acobian algebras from closed surfaces arXiv:1207.3778.

\bibitem{ES08}
K.~Erdmann, A.~Skowro{\'n}ski,
  \href{http://dx.doi.org/10.4171/062-1/5}{Periodic algebras}, in: Trends in
  representation theory of algebras and related topics, EMS Ser. Congr. Rep.,
  Eur. Math. Soc., Z\"urich, 2008, pp. 201--251.
\newblock \href {http://dx.doi.org/10.4171/062-1/5}
  {\path{doi:10.4171/062-1/5}}.
\newline\urlprefix\url{http://dx.doi.org/10.4171/062-1/5}

\bibitem{ES06}
K.~Erdmann, A.~Skowro{\'n}ski,
  \href{http://projecteuclid.org/euclid.jmsj/1145287095}{The stable
  {C}alabi-{Y}au dimension of tame symmetric algebras}, J. Math. Soc. Japan
  58~(1) (2006) 97--128.
\newline\urlprefix\url{http://projecteuclid.org/euclid.jmsj/1145287095}

\bibitem{DWZ08}
H.~Derksen, J.~Weyman, A.~Zelevinsky,
  \href{http://dx.doi.org/10.1007/s00029-008-0057-9}{Quivers with potentials
  and their representations. {I}. {M}utations}, Selecta Math. (N.S.) 14~(1)
  (2008) 59--119.
\newblock \href {http://dx.doi.org/10.1007/s00029-008-0057-9}
  {\path{doi:10.1007/s00029-008-0057-9}}.
\newline\urlprefix\url{http://dx.doi.org/10.1007/s00029-008-0057-9}

\bibitem{FST08}
S.~Fomin, M.~Shapiro, D.~Thurston,
  \href{http://dx.doi.org/10.1007/s11511-008-0030-7}{Cluster algebras and
  triangulated surfaces. {I}. {C}luster complexes}, Acta Math. 201~(1) (2008)
  83--146.
\newblock \href {http://dx.doi.org/10.1007/s11511-008-0030-7}
  {\path{doi:10.1007/s11511-008-0030-7}}.
\newline\urlprefix\url{http://dx.doi.org/10.1007/s11511-008-0030-7}

\bibitem{LF09}
D.~Labardini-Fragoso, \href{http://dx.doi.org/10.1112/plms/pdn051}{Quivers with
  potentials associated to triangulated surfaces}, Proc. Lond. Math. Soc. (3)
  98~(3) (2009) 797--839.
\newblock \href {http://dx.doi.org/10.1112/plms/pdn051}
  {\path{doi:10.1112/plms/pdn051}}.
\newline\urlprefix\url{http://dx.doi.org/10.1112/plms/pdn051}

\bibitem{HI11}
M.~Herschend, O.~Iyama,
  \href{http://dx.doi.org/10.1112/S0010437X11005367}{Selfinjective quivers with
  potential and 2-representation-finite algebras}, Compos. Math. 147~(6) (2011)
  1885--1920.
\newblock \href {http://dx.doi.org/10.1112/S0010437X11005367}
  {\path{doi:10.1112/S0010437X11005367}}.
\newline\urlprefix\url{http://dx.doi.org/10.1112/S0010437X11005367}

\bibitem{Ami12}
C.~Amiot, On generalized cluster categories arXiv:1101.3675.

\bibitem{LF13}
D.~Labardini-Fragoso, On triangulations, quivers with potentials and mutations
  arXiv:1302.1936.

\bibitem{Ke08}
B.~Keller, \href{http://dx.doi.org/10.4171/062-1/11}{Calabi-{Y}au triangulated
  categories}, in: Trends in representation theory of algebras and related
  topics, EMS Ser. Congr. Rep., Eur. Math. Soc., Z\"urich, 2008, pp. 467--489.
\newblock \href {http://dx.doi.org/10.4171/062-1/11}
  {\path{doi:10.4171/062-1/11}}.
\newline\urlprefix\url{http://dx.doi.org/10.4171/062-1/11}

\bibitem{Ga80}
P.~Gabriel, Auslander-{R}eiten sequences and representation-finite algebras,
  in: Representation theory, {I} ({P}roc. {W}orkshop, {C}arleton {U}niv.,
  {O}ttawa, {O}nt., 1979), Vol. 831 of Lecture Notes in Math., Springer,
  Berlin, 1980, pp. 1--71.

\bibitem{Erd90}
K.~Erdmann, Blocks of tame representation type and related algebras, Vol. 1428
  of Lecture Notes in Mathematics, Springer-Verlag, Berlin, 1990.

\bibitem{GKO13}
C.~Geiss, B.~Keller, S.~Oppermann, {$n$}-angulated categories, J. Reine Angew.
  Math. 675 (2013) 101--120.

\bibitem{ARS97}
M.~Auslander, I.~Reiten, S.~S. O., Representation theory of {A}rtin algebras,
  Vol.~36 of Cambridge Studies in Advanced Mathematics, Cambridge University
  Press, Cambridge, 1997, corrected reprint of the 1995 original.

\bibitem{Ami09}
C.~Amiot, \href{http://aif.cedram.org/item?id=AIF_2009__59_6_2525_0}{Cluster
  categories for algebras of global dimension 2 and quivers with potential},
  Ann. Inst. Fourier (Grenoble) 59~(6) (2009) 2525--2590.
\newline\urlprefix\url{http://aif.cedram.org/item?id=AIF_2009__59_6_2525_0}

\bibitem{RVB02}
I.~Reiten, M.~Van~den Bergh,
  \href{http://dx.doi.org/10.1090/S0894-0347-02-00387-9}{Noetherian hereditary
  abelian categories satisfying {S}erre duality}, J. Amer. Math. Soc. 15~(2)
  (2002) 295--366.
\newblock \href {http://dx.doi.org/10.1090/S0894-0347-02-00387-9}
  {\path{doi:10.1090/S0894-0347-02-00387-9}}.
\newline\urlprefix\url{http://dx.doi.org/10.1090/S0894-0347-02-00387-9}

\bibitem{Ri08}
C.~M. Ringel, The self-injective cluster-tilted algebras, Archiv der Mathematik
  91~(3) (2008) 218--225.

\bibitem{KR07}
B.~Keller, I.~Reiten,
  \href{http://dx.doi.org/10.1016/j.aim.2006.07.013}{Cluster-tilted algebras
  are {G}orenstein and stably {C}alabi-{Y}au}, Adv. Math. 211~(1) (2007)
  123--151.
\newblock \href {http://dx.doi.org/10.1016/j.aim.2006.07.013}
  {\path{doi:10.1016/j.aim.2006.07.013}}.
\newline\urlprefix\url{http://dx.doi.org/10.1016/j.aim.2006.07.013}

\bibitem{LF12}
D.~Labardini-Fragoso, Quivers with potentials associated to triangulated
  surfaces, part {IV}: Removing boundary assumptions arXiv:1206.1798v1.

\bibitem{KY11}
B.~Keller, D.~Yang, \href{http://dx.doi.org/10.1016/j.aim.2010.09.019}{Derived
  equivalences from mutations of quivers with potential}, Adv. Math. 226~(3)
  (2011) 2118--2168.
\newblock \href {http://dx.doi.org/10.1016/j.aim.2010.09.019}
  {\path{doi:10.1016/j.aim.2010.09.019}}.
\newline\urlprefix\url{http://dx.doi.org/10.1016/j.aim.2010.09.019}

\bibitem{BIKR08}
I.~Burban, O.~Iyama, B.~Keller, I.~Reiten,
  \href{http://dx.doi.org/10.1016/j.aim.2007.10.007}{Cluster tilting for
  one-dimensional hypersurface singularities}, Adv. Math. 217~(6) (2008)
  2443--2484.
\newblock \href {http://dx.doi.org/10.1016/j.aim.2007.10.007}
  {\path{doi:10.1016/j.aim.2007.10.007}}.
\newline\urlprefix\url{http://dx.doi.org/10.1016/j.aim.2007.10.007}

\bibitem{D80}
J.~A. Drozd, Tame and wild matrix problems, in: Representation theory, {II}
  ({P}roc. {S}econd {I}nternat. {C}onf., {C}arleton {U}niv., {O}ttawa, {O}nt.,
  1979), Vol. 832 of Lecture Notes in Math., Springer, Berlin, 1980, pp.
  242--258.

\bibitem{Sk87}
A.~Skowro{\'n}ski, Group algebras of polynomial growth, manuscripta mathematica
  59~(4) (1987) 499--516.

\bibitem{CB91}
W.~Crawley-Boevey, \href{http://dx.doi.org/10.1112/plms/s3-63.2.241}{Tame
  algebras and generic modules}, Proc. London Math. Soc. (3) 63~(2) (1991)
  241--265.
\newblock \href {http://dx.doi.org/10.1112/plms/s3-63.2.241}
  {\path{doi:10.1112/plms/s3-63.2.241}}.
\newline\urlprefix\url{http://dx.doi.org/10.1112/plms/s3-63.2.241}

\bibitem{GP99}
C.~Geiss, J.~A. de~la Pe{\~n}a, Auslander-{R}eiten components for clans, Bol.
  Soc. Mat. Mexicana (3) 5~(2) (1999) 307--326.

\bibitem{CB89}
W.~Crawley-Boevey, Functorial filtrations \textsc{II}: Clans and the gelfand
  problem, Journal of the London Mathematical Society 2~(1) (1989) 9--30.

\bibitem{Erd88}
K.~Erdmann, \href{http://dx.doi.org/10.1007/BF01194151}{On the number of simple
  modules of certain tame blocks and algebras}, Arch. Math. (Basel) 51~(1)
  (1988) 34--38.
\newblock \href {http://dx.doi.org/10.1007/BF01194151}
  {\path{doi:10.1007/BF01194151}}.
\newline\urlprefix\url{http://dx.doi.org/10.1007/BF01194151}

\bibitem{GeGo13}
C.~Geiss, R.~Gonz{\'a}lez-Silva,
  \href{http://dx.doi.org/10.1007/s10468-014-9486-7}{Tubular {J}acobian
  algebras}, Algebr. Represent. Theory 18~(1) (2015) 161--181.
\newblock \href {http://dx.doi.org/10.1007/s10468-014-9486-7}
  {\path{doi:10.1007/s10468-014-9486-7}}.
\newline\urlprefix\url{http://dx.doi.org/10.1007/s10468-014-9486-7}

\bibitem{BG09}
M.~Barot, C.~Geiss, Tubular cluster algebras {I}: categorification,
  Mathematische Zeitschrift 271~(3-4) (2012) 1091--1115.

\bibitem{Lad13}
S.~Ladkani, \href{arta2013.mat.umk.pl/abstracts/abstr_23.pdf}{On symmetric
  \textsc{J}acobian algebras} Second ARTA conference, Torun 2013.
\newline\urlprefix\url{arta2013.mat.umk.pl/abstracts/abstr_23.pdf}

\end{thebibliography}

\end{document}